\documentclass[12pt]{article}

\usepackage{latexsym}
\usepackage{amsfonts}
\usepackage{amsthm}
\usepackage{amsmath}
\usepackage{amscd}
\usepackage{amssymb}
\usepackage{graphicx}
\usepackage{enumerate}
\usepackage{fancyhdr}

\newtheorem{lem}{Lemma}[section]
\newtheorem{thm}[lem]{Theorem}
\newtheorem{rem}[lem]{Remark}
\newtheorem{prop}[lem]{Proposition}

\newcommand\matF{{\mathbb{F}}}
\newcommand\matN{{\mathbb{N}}}
\newcommand\matZ{{\mathbb{Z}}}
\newcommand\matP{{\mathbb{P}}}
\newcommand\matC{{\mathbb{C}}}
\newcommand\matQ{{\mathbb{Q}}}

\newcommand\matA{{\mathbb{A}}}

\newfont{\Got}{eufm10 scaled 1200}
\newcommand{\permu}{{\hbox{\Got S}}}

\newcommand{\compo}{\,{\scriptstyle\circ}\,}
\newcommand{\Sigmatil}{\widetilde{\Sigma}}
\newcommand{\finedimo}{{\hfill\hbox{$\square$}\vspace{4pt}}}

\begin{document}

\title{On certain permutation groups\\ and sums of two squares}

\author{Pietro~\textsc{Corvaja}
\and Carlo~\textsc{Petronio}
\and Umberto~\textsc{Zannier}}

\maketitle

\begin{abstract}
\noindent
We consider the question of existence of ramified covers over $\matP_1$
matching certain prescribed ramification conditions.
This problem has already been faced in a number of papers, but we discuss alternative
approaches for an existence proof, involving elliptic curves
and universal ramified covers with signature.
We also relate the geometric problem with finite permutation groups and
with the Fermat-Euler Theorem on the representation of a prime as a sum of
two squares.
\end{abstract}

\section*{Introduction}
The present note does not contain new results. It may be considered
as an {\it addendum} to~\cite{PP}, presenting another viewpoint for
one of the results proved there. The paper~\cite{PP} is concerned
with the so-called {\it Hurwitz Existence Problem} for branched
covers between real, closed, and connected surfaces. To any such
cover one can associate certain permutations whose cycle
decompositions lengths are the branching local degrees, and the
problem asks to ``describe'' the cycle lengths coming from an actual
cover. In~\cite{PP} the authors discuss a number of situations, using an
approach based on $2$-orbifolds. More specifically, for various
infinite families of cycles lengths they establish numerical
criteria for the existence of the cover in terms of the (candidate)
total degree.

Thanks to various results obtained over the time (see in particular~\cite{EKS}
and the references quoted in~\cite{PP}), the Hurwitz Existence Problem is only
open when the target of the (candidate) cover is the sphere, that
one can view as the Riemann sphere $\matP_1(\matC)$. Moreover,
by the Riemann Existence Theorem, one knows that any topological branched cover
over the sphere can be realized as
a ramified cover of $\matP_1(\matC)$ by a complex algebraic
curve (see~\cite{GMS, V}). One of our purposes here is precisely to
present one of the conclusions of~\cite{PP} in this last perspective.

We shall focus on an example related with representations of
integers as sums of two squares. It follows directly from~\cite{PP}
that a positive integer $d$ congruent to $1$ modulo $4$ is likewise
representable if and only if there exist three permutations on $d$
letters satisfying certain simple conditions (see
Proposition~\ref{permu:prop} below). In turn, the complex-algebraic viewpoint
will show  that the existence of these permutations is related to the
endomorphisms of the elliptic curve $E$ with Weierstrass equation
$y^2=x^3-x$, namely to the complex torus $E={\matC}/{\matZ}[i]$. The
link with representations of an integer as a sum of two squares
becomes apparent, since the degree of an endomorphism of such an
elliptic curve is always the sum of two squares. This is a
well-known situation, and we also remark that the link already
appears implicitly in~\cite{S} and especially in~\cite{GMS}. (The
latter paper, among many other things, also contains explicit
constructions of algebraic covers associated to the alluded
permutations and several other similar ones.)

The second purpose of this note is to show the following
somewhat surprising fact: putting together some of
the different viewpoints on the Hurwitz Existence Problem
one can get a proof of the Fermat-Euler
theorem, which asserts that a prime $p$ congruent to $1$ modulo $4$
is a sum of two squares. Of course, such a
proof has a small interest in itself, since it relies on deep results,
whereas many elegant and simple proofs are known. However the
connection seems a striking one to us, and it raises the question
whether a direct proof exists in purely combinatorial terms related
to the said permutations.

The paper is organized as follows: in Section~\ref{new:proof:sec} we will recall one of
the results from~\cite{PP},  connect it with the Riemann Existence
Theorem, and  interpret and reprove it in terms of $\mathrm{End}(E)$. Then, in Section~\ref{main:sec},
we will deduce from this connection and interpretation a proof of the Fermat-Euler Theorem.

\section{Certain branched covers\\ of the Riemann Sphere}\label{new:proof:sec}
We start by recalling Theorem 0.4 from~\cite{PP}, that we reformulate using a
slightly different language.
To give the statement we will need the following definition.
Let $\pi: \Sigmatil\to\Sigma$ be a branched topological cover
between real, closed, and connected
surfaces. We say that $\pi$ has
{\it branching type} $(a_1,\ldots ,a_r)$
over a point $P\in\Sigma$, where $a_1,\ldots ,a_r$ are positive integers,
if $\pi^{-1}(P)$ consists of $r$ distinct points
$Q_1,\ldots ,Q_r$ such that locally at $Q_i$ the map $\pi$ may be
represented as $z\mapsto z^{a_i}$, on viewing $\Sigma$ and $\Sigmatil$
as locally homeomorphic to a complex disk.
Henceforth we will  deal with
the case $\Sigma =\matP_1(\matC)$. We have:

\begin{thm}\label{cited:thm}
\cite[Theorem 0.4]{PP}
Suppose $d=4k+1$ for some $k\in\matN$. The following conditions are equivalent:
\begin{itemize}
\item[(I)] There exists a branched cover   $\Sigmatil\to \matP_1(\matC)$ of degree $d$,
ramified over three points, with branching types
$$(1,4,\ldots,4),\qquad (1,4,\ldots ,4),\qquad(1,2,\ldots , 2);$$
\item[(II)] $d=x^2+y^2$ for some $x,y\in\matN$.
\end{itemize}
\end{thm}

\begin{rem}
\emph{For a branched cover $\Sigmatil\to\matP_1(\matC)$ of degree $d=4k+1$ and ramified over three points, the
branching types $(1,4,\ldots,4)$, $(1,4,\ldots ,4)$, and $(1,2,\ldots ,2)$ force $\Sigmatil$
to be $\matP_1(\matC)$ too. In fact the types have lengths $k+1$, $k+1$, and $2k+1$, and
the Riemann-Hurwitz formula shows that if the genus of $\Sigmatil$ is $g$ then
%   $2g-2=-2d+(2k)+(3k)+(3k)=-2d+8k=-2$,
$2(1-g)-(k+1)-(k+1)-(2k+1)=(4k+1)\cdot(2-3)$, which implies that $g=0$.}
\end{rem}

By this remark, from now on we only deal with the case $\Sigmatil=\matP_1(\matC)$.

\begin{rem}
\emph{Up to an automorphism of the target $\matP_1(\matC)$ one can suppose
without loss of generality that, if a branched
cover $\matP_1(\matC)\to\matP_1(\matC)$ has three branching points, these points are
$0$, $1$, and $\infty$. For a cover as in Theorem~\ref{cited:thm} we will always assume
that the branching types are $(1,4,\ldots,4)$ over $0$ and $1$, and $(1,2,\ldots,2)$
over $\infty$.}
\end{rem}

The proofs in~\cite{PP} employ the geometry of 2-obifolds.
For Theorem~\ref{cited:thm} they exploit in particular
the fact that $S^2(4,4,2)$, namely the sphere
with three cone points of orders $4$, $4$, and $2$, bears a Euclidean geometric structure
which is rigid up to rescaling. We will now sketch a proof of Theorem~\ref{cited:thm} in terms of branched
covers of complex algebraic curves. We first note that considering
the monodromy of a branched cover (a representation of the fundamental group
of the complement of the branching points into the symmetric group
$\permu_d$ on the $d$ letters $\{1,\ldots,d\}$), one
gets the following:

\begin{prop}\label{permu:prop}
A cover as in the statement of Theorem~\ref{cited:thm} exists if and only if
there exist permutations $\sigma_0,\sigma_1,\sigma_\infty\in\permu_d$ such that:
\begin{itemize}
\item[(i)] $\sigma_0 \sigma_1=\sigma_\infty$;
\item[(ii)] The cycles in the decompositions of $\sigma_0$ and $\sigma_1$ have lengths
$(1,4,\ldots,4)$, while those in the decomposition of
$\sigma_\infty$ have lengths $(1,2,\ldots ,2)$;
\item[(iii)] The subgroup of $\permu_d$ generated by $\sigma_0,\sigma_1,\sigma_\infty$
is transitive on $\{1,\ldots,d\}$.
\end{itemize}
\end{prop}

This permutation viewpoint was already employed in~\cite{Hurwitz}; see also~\cite{S, V}.
We next spell out the following consequence of the Riemann Existence Theorem already anticipated
in the Introduction:

\begin{prop}\label{RET:prop}
If a topological cover of the sphere onto itself matching
certain branching types exists, it can also be realized
as a cover of algebraic curves $\pi:\matP_1\to\matP_1$,
defined over $\matC$ or even over the algebraic closure $\overline\matQ$ of $\matQ$ in $\matC$.
\end{prop}

Note that a map $\pi$ as in this proposition will be a rational
function with complex coefficients of a complex variable $t$, and
the coefficients may actually be assumed, by specialization or
operating with an automorphism of the domain, to lie in
$\overline{\matQ}$. This last fact implies that the absolute Galois
group of $\overline{\matQ}$ acts on the set of  such covers, leading
to Grothendieck's theory  of {\it dessins d'enfants}, for which the
interested reader is referred to~\cite{Sc}.

\bigskip

Let us now concentrate on branching types as in Theorem~\ref{cited:thm}.
To discuss the existence of a corresponding map $\pi$ as
in Proposition~\ref{RET:prop} we further normalize the situation noting
that, up to composition with an automorphism of the domain $\matP_1$, we
can assume without loss of generality that $0$ (respectively, $1$,
and $\infty$) is the unique unramified point
above $0$ (respectively, $1$, and $\infty$). The branching conditions then
imply that the map $\pi$, if any, can be expressed as
\begin{equation*}    %   \label{pi:shape:eqn}
\pi(t)={tP^4(t)\over R^2(t)}=1+{(t-1)Q^4(t)\over R^2(t)}
\end{equation*}
with $\deg P(t)=\deg Q(t)=k$, and $\deg R(t)=2k$.
More precisely, a cover as in Theorem~\ref{cited:thm} exists
if and only if there exist polynomials $P(t),Q(t),R(t)\in\overline\matQ[t]$
without multiple roots such that
$tP(t)$, $(t-1)Q(t)$, and $R(t)$ are pairwise coprime,
$\deg P(t)=\deg Q(t)=k$, $\deg R(t)=2k$, and
\begin{equation}\label{PQR:eqn}
tP^4(t)=(t-1)Q^4(t)+R^2(t).
\end{equation}

\begin{rem}
\emph{When three such polynomials exist, they
provide an {\it extremal example} of the $abc$-theorem for the field of rational
functions in one variable (see, \emph{e.g.},~\cite{Z1} and ~\cite{Z2}), a known analogue
(due to Mason and Stothers) of the celebrated $abc$-conjecture of
Masser and Oesterl\'e over number fields.}
\end{rem}

\paragraph{Proof of (I)$\;\Rightarrow\;$(II) in Theorem~\ref{cited:thm}.}
Suppose the relevant branched cover exists, so there are polynomials
$P(t)$, $Q(t)$, and $R(t)$ satisfying equation~(\ref{PQR:eqn}) and the other conditions.
Dividing by $Q^4(t)$ in equation (\ref{PQR:eqn})  we obtain the $\overline{\matQ}(t)$-point
$$\left(\frac{P(t)}{Q(t)},\frac{R(t)}{Q^2(t)}\right)$$
on the genus-1 curve over $\matQ(t)$ with affine equation
$y^2=tx^4-(t-1)$. Its points over $\matQ(t)$ form a finitely generated
group, and the involved degrees correspond to the values of a
N\'eron-Tate height; this is a quadratic form, so a connection with
sums of squares begins to emerge. In this particular case the (elliptic) curve
turns out to have constant $j$-invariant (equal to $1728$), so the curve can in fact
be defined over the constant field $\matQ$, which allows one to analyze the situation in
a much simpler way than in more general circumstances.

Indeed, consider the curve which is the normalization of the closure
in $\matP_2$ of the affine curve $v^2=u^4-1$. Since it has genus
$1$, it becomes an elliptic curve $E$ if we choose, \emph{e.g.}, the
point $O:=(0,i)$ as origin (where, here and in the sequel,
$i=\sqrt{-1}$). Note that $E$ is isomorphic, as a complex torus, to
the quotient $\matC/\matZ [i]$, namely it admits an automorphism  of
order four fixing the origin, given by $(u,v)\mapsto (iu,v)$. We shall
prove the following result:

\begin{prop}\label{PQR:give:endo:prop}
Let $P(t),Q(t),R(t)$ be three polynomials
satisfying~(\ref{PQR:eqn}), with $\deg P(t)=\deg Q(t)=k,\, \deg
R(t)=2k$, such that $t\cdot (t-1)\cdot P(t)\cdot Q(t)\cdot R(t)$
has no multiple roots. Then, up to replacing the polynomial $R(t)$ by $-R(t)$,  the map $(u,v)\mapsto (x,y)$, where
$$
x=u{P(t)\over Q(t)},\quad y=v{R(t)\over Q^2(t)},\quad t={u^4\over v^2}
$$
induces an endomorphism of $E$ (as an elliptic curve) of degree $d=4k+1$.
\end{prop}

\begin{proof}
From $v^2=u^4-1$ and $t=u^4/v^2$ it immediately follows that
$u^4=\frac{t}{ t-1}$ and $v^2={1\over t-1}$. Substituting in the
expression for $x$ and $y$ one obtains $x^4={t\over t-1}{P^4(t)\over
Q^4(t)}$ and $y^2={1\over t-1}{R^2(t)\over Q^4(t)}$, which shows that the equality
$x^4-1=y^2$ is equivalent to (\ref{PQR:eqn}). This proves that the
map $(u,v)\mapsto (x,y)$ indeed sends $E$ to itself. Since $x$
vanishes at $O$, the morphism sends $O$ either to itself or to the
point $(0,-i)$, in which case we replace $R(t)$ by $-R(t)$, and
then get that the morphism fixes the origin,
so it is also an endomorphism of $E$ in the sense of elliptic curves.
Its degree is easily seen to be $d=4k+1$.
\end{proof}

Now, condition (II) of the statement follows, since  the ring
$\mathrm{End}(E)$ of the endomorphisms of $E$ (as an elliptic curve)
is well-known to be isomorphic to $\matZ[i]$, with the degree given
by the square of the absolute value. More precisely, the endomorphism in
Proposition~\ref{PQR:give:endo:prop}  corresponds to the multiplication by a Gaussian integer
$a+ib$ and we have $d=a^2+b^2$, concluding the proof. \finedimo

\paragraph{Proof of (II)$\;\Rightarrow\;$(I) in Theorem~\ref{cited:thm}.}
Since for every Gaussian integer $a+ib$ there exists an endomorphism
of $E$ of degree $d=a^2+b^2$, we must show that every endomorphism
$\varphi$ of $E$ of odd degree $d$ can be obtained as above for some
polynomials $P(t)$, $Q(t)$, and $R(t)$. (Similar considerations are
valid for even degrees, leading to slightly different analogue
conclusions.) Let $\varphi\in \mathrm{End}(E)$ be an endomorphism of
degree $d=4k+1$. We have an expression of the form $\varphi(u,v)
=(x(u,v),y(u,v))$ for suitable rational functions $x,y\in \matC(E)$.
Now, the degree-$8$ map $t:E\to\matP_1$ given by $t=u^4/v^2$ is
clearly invariant under the action of the subgroup $G$ of the
automorphisms of $E$ (as algebraic curves) generated by
$$
\alpha:\, (u,v)\mapsto (iu,v),\qquad \beta:\, (u,v)\mapsto(u,-v).
$$
Note that $\alpha$ generates the isotropy group of $O$, whereas
$\beta$ may be also described as the map $p\mapsto \delta -p$ where
$\delta:=(0,-i)$; of course, $\alpha$ has order four and $\beta$ has
order two; note that $\beta$ is central in $G$,
and $\alpha^2\compo\beta: ~p\mapsto p+\delta$ is the (only) central
translation in the automorphism group of $E$. Since $G$ has order $8$, $t$
generates the field of invariants for $G$. On the other hand, it is
easy to check that $t\compo\varphi=x^4/y^2$ is invariant under the
action of $G$, therefore it is a rational function $Z(t)$ of $t$.

The function $t$ has divisor on $E$ of the shape
$4((O)+(\delta))-2((Q_1)+(Q_i)+(Q_{-1})+(Q_{-i}))$, where
$Q_l=(l,0)$. Also, the divisor of $u-i^s$ (for $s=1,\ldots,4$)  is
$2(Q_{i^s})-(\infty_+)-(\infty_-)$, that of $v+ u^2$ is
$2((\infty_+)-(\infty_-))$ for some labeling of the poles of $u,v$,
and finally that of $v-i-u^2$ is $2(\delta)-2(\infty_+)$. It
easily follows that $\delta,\infty_\pm$ have order $2$ on $E$
whereas the $Q_l$'s have order $4$.

With this information, considering zeros and multiplicities, we
see that $Z(t)=x^4/y^2=tP^4(t)/R^2(t)$ for suitable polynomials
$P(t)$ and $R(t)$, where $\deg (tP^4(t))>\deg R^2(t)$ ---here we use the fact
that $d$ is odd, so    $\varphi$ fixes $\delta$ and sends the set of poles of $t$ to itself. Similarly, we have
$x^4/y^2-1=1/y^2$, that we can rewrite as $(t-1)Q^4(t)/R^2(t)$.
Finally, the equation for $E$ shows that
$P(t)$, $Q(t)$, and $R(t)$ satisfy~(\ref{PQR:eqn}) and thus lead to a cover as in part
(I) of the statement.
\finedimo

\begin{rem}
\emph{As a byproduct of our argument we have obtained a correspondence
between permutations as in Proposition~\ref{permu:prop} and polynomials
satisfying relation~(\ref{PQR:eqn}).
Our proof actually also produces a relevant field of definition
for the coefficients, as in~\cite{GMS}.}
\end{rem}

We recall in passing that a Weierstrass model of the curve $E$ employed above is obtained
by the inverse transformations $ \eta:={u\over v-i}$, $\xi:={u^2\over
v-i}=u\eta$ and $u={\xi\over \eta}$, $v=i+{\xi\over \eta^2}$, that
lead to the equation $\eta^2={1\over 2i}(\xi^3-\xi)$.

\paragraph{Galois structure}

We conclude this paragraph with some extra considerations on the
constructions we encountered so far. First of all we prove the
following:

\begin{prop}\label{zGalois:prop}
With the above notation
(in particular $G=\langle\alpha,\beta\rangle$ is the group defined in the previous proof),
the map $t\compo\varphi=Z(t)=:z$ defines a
Galois cover $E\to\matP_1$, whose Galois group $\Gamma$ (of order
$8d$) is
\begin{equation}\label{Gamma:eqn}
\Gamma=\big\{ p\mapsto gp+\kappa:\ g\in G,\ \kappa\in\mathrm{Ker}~\varphi\big\}.
\end{equation}
\end{prop}

We begin with a preliminary result:

\begin{lem}\label{commut:lem}
Let $G=\langle\alpha,\beta\rangle$ be the group defined above. Let
$\varphi:E\rightarrow E$ be an isogeny of odd degree. Then for every
$g\in G$ one has $g\compo\varphi=\varphi\compo g$.
\end{lem}

\begin{proof}
Clearly $\varphi$ commutes with $\alpha$, since both are isogenies.
To prove that $\varphi$ commutes with $\beta$, we shall prove the
equivalent fact that $\varphi$ commutes with $\alpha^2\compo\beta$,
which is the translation by $\delta$. To this end, it is useful to
think in terms of the actions of $\varphi$ and $\alpha^2\compo\beta$
on the complex plane $\matC$. The latter corresponds to the
multiplication by $\frac{1+i}{2}$, while $\varphi$ corresponds to
the multiplication by a Gaussian integer $a+ib$ with $a^2+b^2\equiv
1$ (mod $2$). Now, we have to prove that the functions $\matC\ni
z\mapsto (a+ib)z+\frac{1+i}{2}$ and $\matC\ni z\mapsto
(a+ib)(z+\frac{1+i}{2})$ coincide modulo $\matZ[i]$. This is
equivalent to the fact that
$(a+ib)\frac{1+i}{2}-\frac{1+i}{2}\in\matZ[i]$, which easily follows
from the hypothesis that $a^2+b^2$ is odd.
\end{proof}

\noindent {\it Proof of Proposition~\ref{zGalois:prop}}. The map
$z=t\compo\varphi:E\to\matP_1$  induces a field extension
$\overline{\matQ}(E)/\overline{\matQ}(z)$ of degree $8d$. Let us
prove that it is   invariant under the action of the group $\Gamma$
defined above. Let $p\in E$, $g\in G$ and
$\kappa\in\mathrm{Ker}\varphi$. Clearly,
$\varphi(gp+\kappa)=\varphi(gp)$; now, by Lemma~\ref{commut:lem}, we have
$\varphi(gp)=g(\varphi(p))$ and since $t$ is invariant by $G$ we
have $z(p)=(t\compo\varphi)(p)=z(gp+\kappa)$ as wanted. It remains to
prove that $\Gamma$ has order $8d$; this is due to the fact that
$\varphi$ has odd degree, so $\mathrm{Ker}\varphi$ has odd order.
Then the subgroup of translations in $\Gamma$ has order
$2\deg\varphi$; more precisely  $\Gamma$ is also given as the
extension
$$
\{0\}\to ~\langle\delta\rangle\oplus\mathrm{Ker}~\varphi\to\Gamma\to\{1,i,-1,-i\}\to\{1\},
$$
where the map $\Gamma\rightarrow\{1,i,-1,-i\}$ denotes the action on
the invariant differentials. From this representation, it is clear
that its order is $8d$. Hence $\Gamma$ is the Galois group of the
cover $z=t\compo\varphi: E\to\matP_1$. \finedimo

Proposition~\ref{zGalois:prop} implies
in particular that the Galois closure of the equation in $t$ over
$\overline{\matQ}(z)$ given by $Z(t)=z$ is contained in the above extension
$\overline{\matQ}(E)/\overline{\matQ}(z)$. However the Galois closure, whose Galois
group is generated by the permutations $\sigma_0,\sigma_1,\sigma_\infty$
corresponding to our cover as in Proposition~\ref{permu:prop},
is actually smaller: in
fact, as already noticed, the element $\gamma:=\alpha^2\compo\beta$ acts on $(u,v)$ as
$\gamma(u,v)=(-u,-v)$, namely as a translation by $\delta$, and hence
fixes the field $\overline{\matQ}(t)$. We have also
already remarked that $\gamma$  is in the center of $\Gamma$. Therefore the said Galois
closure is contained in the fixed field of $\gamma$, and is in fact
equal to it, because no  subgroup of $G$ larger then $\langle\gamma\rangle$ is
normal in $\Gamma$. This fixed field of $\gamma$ is easily seen to be
$\overline{\matQ}(u^2,v^2,uv)$. If we set $\sigma:=v/u^3$ and $\tau:=-1/u^2$ we
find that $\sigma$ and $\tau$ generate this field and
$\sigma^2=\tau^3-\tau$. This is a Weierstrass equation for an
elliptic curve $E^*$ (again isomorphic to $E$) which is the quotient of $E$ by the order-2
group of automorphisms generated by $\gamma$.

\begin{rem}
\emph{Equation~(\ref{Gamma:eqn}) yields an explicit representation of the group
generated by our three permutations
$\sigma_0,\sigma_1,\sigma_\infty$, which is isomorphic to
$\Gamma /\langle \gamma\rangle$, and has order $4d$.}
\end{rem}

\begin{rem}
\emph{Alternative proofs based on techniques similar to those employed here are
possible also for Theorems 0.5 and 0.6 in~\cite{PP}.}
\end{rem}

\section{Sums of two squares}\label{main:sec}
To proceed we spell out the following consequence of
the results established in Section~\ref{new:proof:sec}
(and essentially contained in Theorem~\ref{cited:thm}):

\begin{prop}\label{permu:version:prop}
Given a positive integer $d$ congruent to $1$ modulo $4$, there exist permutations
$\sigma_0,\sigma_1,\sigma_\infty\in \permu_d$ satisfying the conditions (i), (ii), and
(iii) of Proposition~\ref{permu:prop}
if and only $d$ is a sum of two squares. If they exist, $\sigma_0,\sigma_1,\sigma_\infty$
generate a group of order $4d$.
\end{prop}

We will now specialize to the case where $d$ is a prime
number and construct the permutations explicitly.

\begin{prop}\label{prime:d:permu:prop}
Let $p$ be a prime congruent to $1$ modulo $4$. Then the group
$\matF_p^*$ has an element $\ell$ of order $4$. Consider the affine
automorphisms $L$ and $T$ of the line $\matA^1$ over $\matF_p$
defined by $L(x)=\ell x$ and $T(x)=x+1$, and the permutations
$$\sigma_0:=L,\qquad\sigma_1= T^{-1}LT,\qquad\sigma_\infty:=LT^{-1}LT$$
of the set $\matF_p=\matA^1(\matF_p)$. Then $\sigma_0,\sigma_1,\sigma_\infty$
satisfy the conditions (i), (ii), and (iii) of Proposition~\ref{permu:prop} with $d=p$.
\end{prop}

\begin{proof}
Existence of $\ell\in\matF_p^*$ of order $4$
readily follows from the assumption $p\equiv 1\ (\mathrm{mod}\ 4)$.
Let us proceed and prove that the permutations $\sigma_0,\sigma_1,\sigma_\infty$ defined in the statement
satisfy the conditions; (i) asserts that $\sigma_\infty=\sigma_0\sigma_1$, which is
indeed true by definition.

The cycle type of $\sigma_0$ is clearly $(1,4,\ldots ,4)$, because
$\ell^2=-1$, whence $L$ and $L^2$ have $0$ as a fixed point and act
injectively on $\matF_p^*$. Since $\sigma_1$ is conjugate to
$\sigma_0$, it also has such a cycle type. Turning to
$\sigma_\infty$, and using again the fact that
$\ell^2=-1$, we see that $\sigma_\infty$ takes
the form $\sigma_\infty(x)=-x+c$, for a suitable $c\in\matF_p$
(actually $c=-\ell-1$). Therefore it is an (affine) involution of
the line $\matA^1$, and since $p\neq 2$ its cycle type is of the
form $(1,2,\ldots ,2)$, which completes the proof of condition (ii).

To establish (iii) we note that the commutator $[\sigma_0,\sigma_1]$
is a nontrivial translation, so it acts transitively on $\matF_p$.
\end{proof}

Combining Propositions~\ref{permu:version:prop} and~\ref{prime:d:permu:prop}
one readily deduces the well-known:

\begin{thm}[Fermat-Euler]\label{FE:thm}
If $p$ is a prime congruent to $1$ modulo $4$
then $p$ is a sum of two squares.
\end{thm}

As mentioned in the Introduction, the resulting proof of this
classical result is extraordinarily demanding: a closer look shows
that, in addition to the construction of the
permutations in Proposition~\ref{prime:d:permu:prop}, it depends also on:

\begin{itemize}
\item[(A)] The topological construction of a finite
cover of $\matP_1(\matC)\setminus\{0,1,\infty\}$ such that lifting three simple disjoint
loops based at a point $P_0$ and encircling $0,1,\infty$ one obtains the given permutations
$\sigma_0,\sigma_1,\sigma_\infty$ on
the fiber over $P_0$. This construction appears, \emph{e.g.},
in~\cite[Theorems 4.27 and 4.31]{V}; it may be proved by patching local
covers or taking a suitable quotient of the universal cover of
$\matP_1(\matC)\setminus\{0,1,\infty\}$.
\item[(B)] The Riemann Existence Theorem, used to realize the said topological
cover as the unramified part of a ramified cover of complex
algebraic curves. This step is delicate and requires
fairly hard analysis, based either on the Dirichlet principle or on the
vanishing of suitable cohomology of holomorphic sheaves on Riemann
surfaces. (See again~\cite[Theorem 4.27]{V}.)
\item[(C)] The structure of the endomorphism ring of the elliptic curve
$E$ of Section~\ref{new:proof:sec}, namely its identification with $\matZ[i]$. This
may be established as follows. First, one notes that $\mathrm{End}(E)$ is a commutative
ring, of rank at most $2$ over $\matZ$: this may be seen by viewing $E$
as $\matC/\Lambda$, where $\Lambda$ is a lattice, and by realizing $\mathrm{End}(E)$
as a ring of multiplications by complex numbers $\mu$ such that
$\mu\Lambda\subset\Lambda$. In our case $\mathrm{End}(E)$ contains a ring
isomorphic to $\matZ[i]$, because it contains $\matZ$ and the order-4 automorphism
denoted by $\alpha$ in Section~\ref{new:proof:sec}. Hence it must be equal to
$\matZ[i]$ because it is commutative, of rank $2$ and integral over
$\matZ$. An alternative way is to note that the elliptic curve $E'$
corresponding to $\matC/\matZ[i]$ has vanishing Weierstrass invariant
$g_3$ (an easy direct computation which amounts to showing that
$\sum_{\omega\in\matZ[i]\setminus\{0\}}\omega^{-6}$ is $0$ ---see~\cite{Sil} for
careful proofs of all of these facts). Hence a Weiestrass equation
for $E'$ has the shape $y^2=4x^3+g_2x$, therefore $E'$ is isomorphic to
$E$, under a substitution $x\mapsto cx$ for suitable $c$.
\end{itemize}

In conclusion, all of these steps involve some nontrivial
mathematics, and (B) is particularly delicate.
Combining the (self-contained) proof of Proposition~\ref{prime:d:permu:prop}
with the arguments used in~\cite{PP} to establish Theorem~\ref{cited:thm},
one gets instead a proof of Theorem~\ref{FE:thm} based on item (A) above and
on the existence and rigidity (up to scaling) of a Euclidean structure on
the orbifold $S^2(4,4,2)$.

On the other hand there exist many few-lines self-contained
proofs of the Fermat-Euler result.
Nevertheless, one cannot say that the proof given above {\it contains},
from the logical viewpoint, any classical proof, as for instance the argument
based on the unique factorization of $\matZ[i]$. In fact, although
this ring plays an implicit role in item (C) above, its factorization
properties are not employed, neither explicitly nor implicitly.

\paragraph{Ramified covers with signature}
An alternative approach to Theorem~\ref{cited:thm}, to which the argument in~\cite{PP}
is closer and which does not require items (B) and (C) above, is described in
a sketchy but complete fashion in~\cite[pp. 60-63]{S}.
This avoids the viewpoint of complex algebraic curve altogether,
being based on the notion of {\it ramified cover
with signature} which, roughly speaking, consists of a
topological cover of a space deprived of finitely many points,
together with a ramified structure above the remaining points, of
the same type as a map of the shape $z\mapsto z^n$.

In our case we have a {\it universal covering with signature
$(4,4,2)$}, meaning a space $Y$ which is obtained by suitably
completing the quotient of the universal cover of $\matP_1(\matC)\setminus
\{0,1,\infty\}$ by the normal subgroup $N$ of $\pi_1(\matP_1(\matC)\setminus
\{0,1,\infty\})$ generated by $c_0^4,c_1^4,c_\infty^2$, where
$c_0,c_1,c_\infty$ are the simple disjoint loops already mentioned above. As
stated in~\cite[p. 63]{S}, one realizes $Y$
as the Euclidean plane $\matC$,
with covering group $\Gamma$ given by the rigid motions of the plane
preserving the orientation and the lattice $\matZ[i]$; we have
$\Gamma\cong \pi_1(\matP_1(\matC)\setminus \{0,1,\infty\})/N$.

An explicit description of the elements of $\Gamma$ as affine transformations of the complex line is as follows:
\begin{equation}\label{Gamma}
\Gamma=\Big\{\matC\ni z\mapsto u z+\lambda\, :\, u\in\{1,i,-1,-i\},\, \lambda\in \matZ[i]\Big\}.
\end{equation}
There are three orbits of points in $\matC$ having non trivial
stabilizers: the first one is $\matZ[i]$, where each point has a
stabilizer of order four; the second one is $\frac{1+i}{2}+\matZ[i]$,
also having a stabilizer of order four; the third one is
$\left(\frac{1}{2}+\matZ[i]\right)\cup\left(\frac{i}{2}+\matZ[i]\right)$, having a
stabilizer of order two. They correspond to the pre-images of
$0,1,\infty$.

\begin{thm}\label{universale}
 The group $\Gamma$ defined by (\ref{Gamma}) is the universal group of type $(4,4,2)$.
\end{thm}

\begin{proof}
We confine ourselves to a sketch. Let $G$ be a group of type  $(4,4,2)$, so $G$
is generated by two elements $\alpha,\beta$ with
$\alpha^4=\beta^4=(\alpha\beta)^2=1$. It is immediate from this
presentation that the congruence class modulo four of the length of a word
representing an element of $G$ only depends on that element.  Also,
it is easily checked that the words of length divisible by four
commute and can be generated by $u:=\alpha^3\beta$ and
$v:=\alpha\beta^3$. Hence we always have a group homomorphism
$G\rightarrow\matZ/4\matZ$ whose kernel is an Abelian (normal)
subgroup generated by two elements. So we have an exact sequence
$\{0\}\to\langle u,v\rangle\to G\to\matZ/4\matZ$, where $\langle
u,v\rangle \cong (\matZ/a\matZ)\times(\matZ/b\matZ)$ for integers
$a,b$  (possibly zero or one) and the last morphism (which need not
be surjective) is the reduction modulo four of the word length.
We note that the action (by conjugation) of $G$ on $\langle u,v\rangle$ is
uniquely determined by the initial relations $\alpha^4=\beta^4=(\alpha\beta)^2=1$.
Coming back to our group $\Gamma$, let us consider the three
elements $c_0,c_1,c_\infty$ of $\Gamma$ defined by  $c_0(z)=iz$,
$c_1(z)=1+iz$, $c_\infty(z)=-z+i$. Then $c_0^4=c_1^4=c_\infty^2=1$
and $c_0c_1=c_\infty$. Clearly, $\tilde{u}:=c_0^3c_1$ and $\tilde{v}:=c_0c_1^3$ act as  $\tilde{u}(z)=z+i$ and
$\tilde{v}(z)=z-1$, so they generate the subgroup of translations,
isomorphic to $\matZ[i]\cong \matZ^2$. Hence we have the exact
sequence
\begin{equation*}
 \{0\}\to\matZ^2\to\Gamma\to\matZ/4\matZ\to\{0\}.
\end{equation*}
Note that the action of $\Gamma$ on $\matZ [i]$ is compatible with the
action of $G$ on $\langle u,v\rangle$, in the natural sense: for instance we have
in $G$ the relation $\alpha u\alpha^{-1}=v^{-1}$, which corresponds in $\Gamma$ to the
relation $c_0\tilde{u}c_0^{-1}=\tilde{v}^{-1}$. From this fact it
easily follows that $G$ is a quotient of $\Gamma$.
\end{proof}

Let us get back to the setting of Theorem~\ref{cited:thm},
and let us use the universal covering with signature $(4,4,2)$, denoted
by $Y$ as above. A cover $X$ of the Riemann sphere with the relevant branching types
exists if and only if it can be realized as the quotient of $Y$ by a
subgroup $\Delta$ of $\Gamma$, of {\it odd} index $d$ in $\Gamma$. The
permutations $\sigma_0,\sigma_1,\sigma_\infty$
then correspond to the images of $c_0,c_1,c_\infty$ in
the permutation representation of $\Gamma$ on the right cosets
$\Gamma/\Delta$. One easily sees that if $\Delta$ exists then it must
contains an element $\sigma$ of order $4$, which must be a rotation
of $\pi/2$ around some point. The orbit of the origin by $\Delta$
is a lattice stable under $\sigma$, which corresponds to an ideal in
$\matZ[i]$. This ideal is principal, and we find again the
conclusion that $d$ is a sum of two squares.
As a matter of fact, to conclude one can also
avoid invoking the principal-ideal ring structure,
by observing that the said lattice
must have a basis of type $v,\sigma v$, whence its index is
necessarily a sum of two squares.

\begin{rem}
\emph{When the degree is an odd prime $p$, this approach also allows
one to elucidate the structure of the group $G$ generated by the
permutations $\sigma_0,\sigma_1,\sigma_\infty$. In fact, as stated
in the proof of Theorem~\ref{universale}, the group generated by the
words of length $4$ in $\sigma_0,\sigma_1$ is commutative. Hence $G$
has an Abelian subgroup $G_0$ of index at most $4$. Since $G$ is a
transitive subgroup of $\permu_p$ it contains a $p$-cycle $g$, which
must lie in $G_0$. Then $G_0$ must be the group generated by $g$.}
\end{rem}

\vspace{1cm}

\noindent
Dipartimento di Matematica e Informatica\\
Universit\`a di Udine\\
Via delle Scienze, 206\\
33100 UDINE --  Italy\\
pietro.corvaja@dimi.uniud.it\\

\vspace{1cm}

\noindent
Dipartimento di Matematica Applicata\\
Universit\`a di Pisa\\
Via Filippo Buonarroti, 1C\\
56127 PISA -- Italy\\
petronio@dm.unipi.it\\

\vspace{1cm}

\noindent
Scuola Normale Superiore\\
Piazza dei Cavalieri, 7\\
56126 PISA -- Italy\\
u.zannier@sns.it

\end{document}